\documentclass[12pt]{amsart}

\usepackage{amssymb}
\usepackage{amsfonts}
\usepackage{graphicx}

\newcommand{\Rn}{{\mathbf R}^{n}}
\newcommand{\Rq}{{\mathbf R}^{4}}
\newcommand{\Rt}{{\mathbf R}^{3}}
\newcommand{\Rd}{{\mathbf R}^{2}}

\newcommand{\de}{\text{\rm d}\mspace{1mu}}

\newcommand{\fff}{\boldsymbol{I}}
\newcommand{\sff}{\mathbf{I} \mathbf{I}}
\newcommand{\He}{\text{Hess}\mspace{1mu}}
\newcommand{\rank}{\rm rank\,}
\newcommand{\beq}{\begin{equation}}
\newcommand{\eeq}{\end{equation}}
\newcommand{\ba}{\begin{align}}
\newcommand{\ea}{\end{align}}
\newcommand{\baa}{\begin{align*}}
\newcommand{\eaa}{\end{align*}}
\newcommand{\lra}{\longrightarrow}

\newtheorem*{theo}{Theorem}
\newtheorem{theorem}{Theorem}
\newtheorem{lemma}{Lemma}
\newtheorem{cor}{Corollary}
\newtheorem{prop}{Proposition}

\newtheorem*{lnf}{Local Normal Form}

\theoremstyle{definition}
\newtheorem{defn}{Definition}
\newtheorem{ex}{Example}

\theoremstyle{remark}
\newtheorem{rem}{Remark}

\begin{document}

\title{The Gauss map for Lagrangean and isoclinic surfaces}

\author{J. Basto-Gon\c calves}
\date{\today}
\address{Centro de Matem\'atica da Universidade do
Porto, Portugal}
\email{jbg@fc.up.pt}
\thanks{The research of the  author at Centro de Matem\'atica da Universidade do Porto (CMUP) was
funded by the European Regional Development Funding FEDER through the programme COMPETE and by the Portuguese Government through the FCT -- Funda\c{c}\~ao para a Ci\^encia e a
Tecnologia under the project PEst-C/MAT/UI0144/2011, and  also by Calouste Gulbenkian
Foundation}

\subjclass{Primary:}

\begin{abstract}
Simple properties of the Gauss map characterise important classes of surfaces in $\Rq$: $R$-surfaces, the real version of plane complex curves; Lagrangean surfaces; isoclinic surfaces.
\end{abstract}
\maketitle

\section{Introduction}

The origin of the study of surfaces  in $\Rq$ was the interpretation of a complex plane curve as a real surface. In a different context, an invariant torus for a two degrees of freedom integrable Hamiltonian system is also a surface  in $\Rq$.

 These surfaces have very special properties, and our aim is to consider a bigger class of surfaces, containing all those, but retaining some of their properties, at least in a partial form.

If a surface is a real version of a complex plane curve, which we will call a $R$-surface \cite{kom}, there are three properties upon which we will focus:
\begin{itemize}
\item All planes are isoclinic to each other
\item The Gauss map has a constant component, in sphere coordinates.
\item The surface is a Lagrangean submanifold of $\Rq$.
\end{itemize}

When considering two planes $P$ and $P'$ in $\Rq$, passing though the origin, we can take the angle that a unit vector $u$ in $P$ makes with its orthogonal projection $u'$ in $P'$; when $u$ describes circumference of radius 1 centred at the origin, that angle varies between two extreme values, in general different. The planes are said to be isoclinic to each other if in fact the angle remains constant.

This can be seen in different way: the unit circumference in $P$ projects as an ellipse in $P'$, and the axes correspond to the extreme values of the angle above; thus the planes are isoclinic to each other when the ellipse is a circumference.

The Gauss map associates to a point in the surface its tangent plane seen as a vector subspace of $\Rq$; this is a point in the Grassmannian $G_{2,4}$. It is well known that $G_{2,4}\cong \mathbf S^2\times \mathbf S^2$, thus the Gauss map is a pair of maps into spheres, and for a $R$-surface one of them is constant.

To define a Lagrangean submanifold we need to specify a symplectic structure, and in fact there are two natural choices for a $R$-surface, as will be seen in subsection~\ref{Rsurf}.

We will consider a  $L$-surface as a surface that is congruent to a Lagrangean surface with the standard symplectic form, and an $I$-surface, or isoclinic surface, a surface for which the Gauss curvature and the normal curvature coincide, up to sign. The $I$-surfaces can also be characterized for having an isoclinic tangent direction at every point \cite{w}.

Under very general hypotheses, given an initial curve, the Gauss map $\Gamma$ allows the reconstruction of the surface which contains the curve and whose Gauss map is $\Gamma$ \cite{weiner}. This involves solving a partial differential equation.

Our aim here is just the characterization of the surface by simple properties of the Gauss map, and the main result is:

\begin{theo}
A surface $S\subset \Rq$ with Gauss map $\Gamma : S\lra G_{2,4}\cong \mathbf S^2\times \mathbf S^2$ is a $L$-surface if, and only if,  the image of $\Gamma_1$, or  $\Gamma_2$, be contained in a great circle.
\end{theo}

\newpage
\section{Basic definitions and results}

We consider a surface $S\subset \mathbf R^4$ locally given by a parametrization:
\[
\Xi : U\subset\mathbf R^2\longrightarrow \Rq
\]
and  a set $\{e_1,e_2,e_3,e_4\}$  of orthonormal vectors, depending on $(x,y)\in U$,  satisfying:
\begin{itemize}
\item $e_1(x,y)$ and $e_2(x,y)$ span the tangent space $T_{\Xi(x,y)}S$ of $S$ at $\Xi(x,y)$.
\item $e_3(x,y)$ and $e_4(x,y)$ span the normal space $N_{\Xi(x,y)}S$ of $S$ at $\Xi(x,y)$.
\end{itemize}
Then $\Xi,\{e_1,e_2,e_3,e_4\}$ is an \emph{adapted moving frame} for $S$. 

While the image of $D\Xi$ is the tangent space of $S$, the image of the second derivative $D^2\Xi$ has both tangent and normal components; the vector valued quadratic form associated to the normal component:
\beq
(D^2\Xi\cdot e_3) e_3+(D^2\Xi\cdot e_4) e_4
\eeq
is the \emph{second fundamental form} $\sff$ of $S$. It can be written \cite{little} as $\sff_1e_3+\sff_2 e_4$, with:
\begin{subequations}\label{sff}
\begin{align}
\sff_1=&au_1^2+2bu_1u_2+cu_2^2\label{sff1}\\
\sff_2=&eu_1^2+2fu_1u_2+gu_2^2\label{sff2}
\end{align}
\end{subequations}
where $u=u_1 e_1+u_2 e_2$ is a tangent vector.

We can express the Gaussian, normal and mean curvature in terms of the coefficients of the second fundamental form \cite{little}:
\begin{align}\label{curvatures}
K=&(ac-b^2)+(eg-f^2)\\
\kappa=&(a-c)f-(e-g)b\notag\\
\mathcal H=&\dfrac{1}{2}(a+c)e_3+\dfrac{1}{2}(e+g)e_4\notag
\end{align}

Assume the surface $S$ locally given by a  parametrization:
\[
\Xi : (x,y)\mapsto (x,y,\varphi(x,y),\psi(x,y))
\]
where $\Phi=(\varphi, \psi)$ has vanishing first jet at the origin,  $j^1\Phi(0)=0$.

The vectors $T_1$ and $T_2$ span the tangent space of $S$:
\[
T_1=\Xi_x=(1,0,\varphi_x,\psi_x),\quad T_2=\Xi_y=(0,1,\varphi_y,\psi_y)
\]
the index $z$ standing for derivative with respect to $z$.

The induced metric in $S$ is given by the first fundamental form:
\[
\fff=E\de x^2+2F\de x\de y+G\de y^2
\]
where:
\[
E=T_1\cdot T_1,\quad F=T_1\cdot T_2,\quad G=T_2\cdot T_2,
\]
We define:
\[
W=EG-F^2
\]

Instead of an orthonormal frame, it is more convenient to take a basis:
\begin{align}
T_1=&(1,0,\varphi_x,\psi_x),& & T_2=(0,1,\varphi_y,\psi_y)\\
\notag
N_1=&(-\varphi_x,-\varphi_y,1,0), && N_2=(-\psi_x,-\psi_y,0,1)
\end{align}
The vectors $T_1$ and $T_2$ span the tangent space, and the vectors $N_1$ and $N_2$ span the normal space. We define:
\[
\hat E=N_1\cdot N_1,\quad \hat F=N_1\cdot N_2,\quad \hat G=N_2\cdot N_2,
\]
and it is easy to verify that:
\[
\hat E  \hat G-\hat F^2=W
\]

\begin{prop}
The Gaussian curvature is  given by:
\beq
K=\dfrac{1}{W^2}(\hat E H_\psi-\hat FQ+\hat G H_\varphi)\label{K}
\eeq
where:
\[
H_f=\He(f)=\left|
\begin{array}{cc}
f_{xx} &f_{xy}\\
f_{xy} & f_{yy}
\end{array}
\right|, \quad
Q=\left|
\begin{array}{cc}
\varphi_{xx} & \varphi_{xy}\\
\psi_{xy} & \psi_{yy}
\end{array}
\right|-\left|
\begin{array}{cc}
\varphi_{xy} & \varphi_{yy}\\
\psi_{xx} & \psi_{xy}
\end{array}\right|
\]
\end{prop}

\begin{prop}
The  normal curvature is given by:
\beq
\kappa=\dfrac{1}{W^2}(EL-FM+GN)\label{N}\
\eeq
where:
\[
L=\left|
\begin{array}{cc}
\varphi_{xy} & \varphi_{yy}\\
\psi_{xy} & \psi_{yy}
\end{array}
\right|, \quad
M=\left|
\begin{array}{cc}
\varphi_{xx} & \varphi_{yy}\\
\psi_{xx} & \psi_{yy}
\end{array}
\right|, \quad
N=\left|
\begin{array}{cc}
\varphi_{xx} & \varphi_{xy}\\
\psi_{xx} & \psi_{xy}
\end{array}
\right|
\]
\end{prop}

\newpage
\section{The Gauss map}

For an immersed  oriented two dimensional surface $S$ in $\Rq$, the \emph{Gauss map} takes a point $p$ of the surface $S$ to the oriented 
plane through the origin $[T_pS]$ which is parallel to the tangent plane $T_pS$ to $S$ at $p$. 
The image of this map is a subset of the Grassmannian manifold $G_{2, 4}$, the  set of oriented two planes in $\Rq$. 

In $\Rq$ with oriented basis given by $e_i$, 
$i=1,\ldots,4$, we can completely determine an oriented plane $P$, defined by the ordered pair of perpendicular unit vectors $v_1=c_1 e_1 + c_2 e_2 + c_3 e_3 + c_4 e_4$ 
and $v_2=d_1 e_1 + d_2 e_2 + d_3 e_3 + d_4 e_4$, by giving 
the algebraic areas $p_{ij}=c_id_j-c_jd_i$ of the projections of the unit square associated to $v_1,\ v_2$ in that plane 
into the six coordinate planes $e_i\wedge e_j$, $i<j$:
\[
P=p_{12}e_1\wedge e_2+p_{13}e_1\wedge e_3+p_{14}e_1\wedge e_4+p_{34}e_3\wedge e_4+p_{42}e_4\wedge e_2+p_{23}e_2\wedge e_3
\]
The six Pl\"ucker coordinates are $P=(p_{12},p_{13},p_{14},p_{34},p_{42},p_{23})$, and could be obtained by taking the wedge product of 
any oriented pair of linearly independent vectors $v_1$ and $v_2$:
\[
P=\dfrac{v_1\wedge v_2}{|v_1\wedge v_2|}
\]

The set of two planes in four space is not six dimensional since 
there are two relations among these Pl\"ucker coordinates, as can be shown by direct calculation:
\begin{itemize}
\item $\sum_{i<j}p_{ij}^2=1$.
\item $p_{12}p_{34}+p_{13}p_{42}+p_{14}p_{23}=0$.
\end{itemize}
The first relation defines $\mathbf S^5\subset \mathbf R^6$, and the two together the Klein quadric.

We define the Klein coordinates by:
\begin{align}
a_1& = p_{12}+p_{34}, &a_2& =p_{13}+p_{42}, &a_3 &= p_{14}+p_{23}\\
\notag
b_1& = p_{12}-p_{34}, &b_2 &=p_{13}-p_{42}, &b_3 &= p_{14}-p_{23}
\end{align}

The two algebraic conditions on the Pl\"ucker coordinates then become:
\[
a_1^2+a_2^2+a_3^2=1,\quad b_1^2+b_2^2+b_3^2=1
\]

Thus the vectors $a = (a_1, a_2, a_3)$ and $b = (b_1 b_2, b_3)$ are unit vectors in three space, 
the sphere coordinates for a plane in $G_{2, 4}$, and we can view 
$G_{2, 4}$ as the product of two 2-spheres:
\[
G_{2, 4}\cong \mathbf S^2\times \mathbf S^2
\]
The Gauss map $p\mapsto [T_p S]$ will be denoted by:
\[
\Gamma : S\longrightarrow G_{2, 4}, \quad \Gamma=(\Gamma_1,\Gamma_2), \quad \Gamma_{1,2}:
S\longrightarrow \mathbf S^2
\]
We shall use $\Gamma$  with the Pl\"ucker coordinates, and $(\Gamma_1,\Gamma_2)$ with the Klein coordinates.

It is possible to define a symplectic form (or an area form) $\sigma$ on $\mathbf S^2\subset \mathbf R^3$ as follows:
\beq\label{sigma}
\sigma=i^* \zeta , \quad \zeta =x\de y\wedge\de z+y\de z\wedge\de x+z\de x\wedge\de y
\eeq
where $i: \mathbf S^2\lra\mathbf R^3$ is the standard  inclusion.

If we consider a surface $S$ locally given by a parametrization:
\[
\Xi : (x,y)\mapsto (x,y,\varphi(x,y),\psi(x,y))
\]
where the 1-jets of $\varphi$ and $\psi$ vanish at the origin, then:
\[
T_1\wedge T_2=(1,\varphi_y,\psi_y,\varphi_x\psi_y- \varphi_y\psi_x,\psi_x,-\varphi_x), \quad |T_1\wedge T_2|=\sqrt{W}
\]
The Gauss map is given by:
\begin{align}
\Gamma_1=\dfrac{1}{\sqrt{W}}(1+\varphi_x\psi_y- \varphi_y\psi_x,-\varphi_x+\psi_y , -\varphi_y-\psi_x) \\
\notag
 \Gamma_2=\dfrac{1}{\sqrt{W}} (1-\varphi_x\psi_y+\varphi_y\psi_x,-\varphi_x-\psi_y , -\varphi_y+\psi_x)
\end{align}
and the normal space is represented by $(\Gamma_1,-\Gamma_2)$.

There is a relation between the Gauss map and the Gaussian and normal curvature of the surface, as expressed by the following:

\begin{theorem}[Blaschke \cite{bla}]\label{primus}
The pull-back of the form $\sigma$ on $ \mathbf S^2$ by a component of the Gauss map is:
\beq\label{curvaturas}
\Gamma_1^* \sigma =(K+\kappa)\,\omega_1\wedge \omega_2, \quad
\Gamma_2^* \sigma =(K-\kappa)\,\omega_1\wedge \omega_2
\eeq
\end{theorem}

\begin{rem}
The above statement is equivalent to:
\[
\det {\rm D}\Gamma_1=K+\kappa,\quad \det {\rm D}\Gamma_2=K-\kappa 
\]
\end{rem}

The Grassmannian $G_{2, 4}$ can also be considered a symplectic manifold in two natural ways, with symplectic forms $\Omega_{\pm}$ defined as follows:
\[
\Omega_{+}=\dfrac{1}{2}\left(\pi_{a}^*\sigma_{a}+\pi_{b}^*\sigma_{b}\right), \quad
\Omega_{-}=\dfrac{1}{2}\left(\pi_{a}^*\sigma_{a}-\pi_{b}^*\sigma_{b}\right)
\]

\begin{cor}
The pull-back of the symplectic forms $\Omega_{\pm}$ on $G_{2, 4}$ by  the Gauss map is:
\beq\label{curvaturas'}
\Gamma^* \Omega_{+} =K\,\omega_1\wedge \omega_2, \quad
\Gamma^* \Omega_{-}=\kappa\,\omega_1\wedge \omega_2
\eeq
\end{cor}

\newpage

\section{Isoclinic surfaces}

Wong~ \cite{w} developed a curvature theory for surfaces 
in $\Rq$ based on the two angles between two tangent planes of the surface. We are interested in conditions under which those two angles be equal:

\begin{defn}
Two planes in $\Rq$, passing through the origin, are isoclinic if the the unit circumference centered at the origin in one of them projects as a circumference in the other one.
\end{defn}

\begin{defn}
The isocline surface $\mathcal I_P$ of a plane $P$ in $\Rq$ is the subset of $G_{2,4}$ consisting of the planes $P'$ that are isoclinic to $P$.
\end{defn}

We take  the $(x,y)$-plane, defined by $u=0$, $v=0$. A plane $P$ defined by:
\[
u=\alpha_1x+\beta_1y,\quad v=\alpha_2x+\beta_2y
\]
is isoclinic to the $(x,y)$-plane if the equation:
\[
x^2+y^2+(\alpha_1x+\beta_1y)^2+(\alpha_2x+\beta_2y)^2=1
\]
defines a circumference in the $(x,y)$-plane. This is equivalent to:
\begin{equation}\label{isosup}
\|\alpha\|^2=\|\beta\|^2, \quad \alpha\cdot \beta=0
\end{equation}
where $\alpha=(\alpha_1,\alpha_2)$ and $\beta=(\beta_1,\beta_2)$. The equations \eqref{isosup} for $(\alpha,\beta)$ define a singular surface $\mathcal I$ in $\Rq$; the unique singularity    is at the origin $(0,0,0,0)$, corresponding to the  $(x,y)$-plane.

We can view $\mathcal I$ as the union $\mathcal I=\mathcal I^+\cup  \mathcal I^-$ of two orthogonal planes in $\Rq$:
\[
\mathcal I^{\pm}: \qquad \beta_1= \mp\alpha_2,\quad  \beta_2=\pm \alpha_1
\]
Let $C$ be the linear map exchanging the third and fourth coordinates:
\[
C=\left[\begin{matrix}
1&0&0&0\\
0&1&0&0\\
0&0&0&1\\
0&0&1&0\\
\end{matrix}\right]
\]
Then:
\[
\mathcal I^{+}=C\mathcal I^{-}
\]

The above is a parametrization of a part of the isocline surface. In fact, the isocline surface $\mathcal I$ of the $(x,y)$ plane is a compact subset of $G_{2,4}$, the union of the two sets:
\[
\mathcal I^+=\{P\in G_{2,4} \mid a(P)=e_1\in S^2\},\quad \mathcal I^-=\{P\in G_{2,4} \mid b(P)=e_1\in S^2\}
\]
where $(a(P), b(P))$ denotes the  sphere coordinates on $G_{2,4}$.

\begin{prop}
The isocline surface $\mathcal I_P$ of the plane $P$ is the union of the orbit of $P$ under the action of $SU(2)$ in $\Rq$ and of its image under the linear map $C$ :
\[
\mathcal I_P=\{AP \mid A\in SU(2)\}\cup \{CAP \mid A\in SU(2)\}
\]
\end{prop}

\begin{defn}
Let $u\in T_pS$ with $\|u\|=1$; then $u$ spans an isoclinic direction of type $+$ ($-$) if there exists   a curve $\gamma (t)$ on $S$ such that:
\[
\gamma(0)=p,\quad \dot\gamma(0)=u, \qquad \dot P(0)\in T_P\mathcal I_P^{+} \ (T_P\mathcal I_P^{-})
\]
with $P(t)=T_{\gamma (t)}S$, $P=P(0)=T_pS$.
\end{defn}

\begin{theorem}[Wong \cite{w}]
There exists an isoclinic direction at a given  point $p\in S$ if, and only if:
\begin{equation}
|K(p)|=|\kappa(p)|
\end{equation}
\end{theorem}

\begin{prop}
The isoclinic direction is defined by:
\[
(a\pm f)\omega_1+(b\pm g)\omega_2=0,\quad \hbox{as } K=\pm \kappa
\]
\end{prop}

\begin{defn}
The surface $S$ is isoclinic , or a $I$-isurface, f there exists (at least) an isoclinic direction at every point of $S$.
\end{defn}

We can consider a line field $\mathfrak I$ on  any isoclinic surface, taking the isoclinic direction at every point, and also a particular vector field $X_{\mathcal I}$ spanning it, given by:
\[
X_{\mathcal I}\;{\rfloor}\ \omega_1\wedge \omega_2=(a\pm f)\,\omega_1+(b\pm g)\,\omega_2
\]
We consider the $+$ case, when $K= \kappa$; the other one is similar.

\begin{theorem}
The vector field $X_{\mathcal I}$ is symplectic.
\end{theorem}
\begin{proof}
We have to prove that $(a+ f)\,\omega_1+(b+ g)\,\omega_2$ is closed. Using the Monge normal form around an arbitrary $p\in S$ we see that:
\[
(a+ f)\,\omega_1+(b+ g)\,\omega_2=(\varphi_{xx}+\psi_{xy})\de x+(\varphi_{xy}+\psi_{yy})\de y+O(2)
\]
and as the exterior derivative at a point only depends on the coefficients and its derivatives at that point:
\[
\de \left[(a+ f)\,\omega_1+(b+g)\,\omega_2\right]=\de \left[(\varphi_{xx}+\psi_{xy})\de x+(\varphi_{xy}+\psi_{yy})\de y\right] =0
\]
at the origin, corresponding to $p\in S$.
\end{proof}

\newpage
\section{Lagrangean surfaces}

A \emph{symplectic manifold} is a pair $(M,\omega)$, where $M$ is a $2n-$dimensional differentiable manifold and $\omega$ is a symplectic form: a closed non degenerate $2$-form. Then:
\[
\Omega=\dfrac{1}{n!}\,\omega^n \hbox{ \  is a volume form, and } \de\mspace{1mu}\omega=0
\]
A {\sl symplectic map} is  a map $\varphi:(M,\omega)\longrightarrow (M',\omega')$, such that:
\[
\varphi^*\omega' =\omega
\]

A  \emph{Lagrangean submanifold} $L$ of $(M,\omega)$ is an immersed submanifold of $M$ such that:
\[
i^* \omega\equiv 0,\quad \hbox{where }i:L\longrightarrow M \hbox{ is the immersion map}
\]

\begin{defn}
A \emph{Lagrangean surface} $\mathcal L$ is an immersed Lagrangean submanifold of $(\Rq,\omega)$.\end{defn}

\begin{lnf}
Given $p\in \mathcal L$, there is a  change of coordinates, by a translation and a linear symplectic and orthogonal map, such that locally $\mathcal L$ becomes the graph around the origin of a map 
\[
\Phi=(\varphi,\psi): \Rd \longrightarrow\Rd
\]
satisfying:
\begin{itemize}
\item The first jet of $\Phi$  is zero at the origin.
\item $\dfrac{\partial\varphi}{\partial y}\equiv\dfrac{\partial\psi}{\partial x}$
\end{itemize}
\end{lnf}

\begin{rem}
If we preserve orientation, so that the linear map $A\in SO(4)$, there is another normal form; the symplectic form in $\Rq$ is then $\omega'=
\de x\wedge \de u-\de y\wedge \de v$ and the identity in the normal form is:
\[
\dfrac{\partial\varphi}{\partial y}\equiv-\dfrac{\partial\psi}{\partial x}
\]
\end{rem}

\begin{prop}
A necessary  condition for $\mathcal L\subset \Rq$ to be a Lagrangean surface is that the Gaussian curvature and the normal curvature coincide up to sign:
\[
|K|\equiv |\kappa|
\]
\end{prop}
\begin{proof}
Assume that $\mathcal L\subset \Rq$ is Lagrangean, and locally given by a parametrization as in the normal form:
\[
(x,y)\mapsto (x,y,\varphi(x,y),\psi(x,y)),\quad \varphi_y\equiv \psi_x
\]
where $\Phi=(\varphi, \psi)$ has vanishing first jet at the origin,  $j^1\Phi(0)=0$.

It follows from $\varphi_y\equiv \psi_x$ that in the moving frame associated to the local normal form we have:
\beq
\hat E=E,\quad \hat F=F,\quad \hat G=G, \qquad L=H_{\psi},\quad M=Q,\quad N=H_{\varphi}\label{EE}
\eeq
and therefore:
\[
K\equiv \kappa
\]
If instead we have $\varphi_y\equiv -\psi_x$,  then in the moving frame associated to the local normal form:
\[
\hat E=E,\quad \hat F=-F,\quad \hat G=G, \qquad L=-H_{\varphi},\quad M=Q,\quad N=-H_{\psi}
\]
and so:
\[
K\equiv -\kappa
\]

\end{proof}

\begin{rem}
If $\mathcal L\subset \Rq$ is a Lagrangean surface, in the moving frame associated to the local normal form we also have:
\beq
e\equiv b,\quad f\equiv c\label{e=b}
\eeq
\end{rem}

The condition $|K|\equiv |\kappa|$ is not a sufficient condition for  the surface to be Lagrangean, as there are rigid motions that do not preserve the symplectic form while naturally preserving the Gauss curvature and the normal curvature. The example~\ref{notsuf} shows that $|K|\equiv |\kappa|$ is not a sufficient condition for  the surface to be  congruent to a Lagrangean surface.

\begin{cor}
A necessary  condition for $\mathcal L\subset \Rq$ to be a Lagrangean surface is that it be an isoclinic surface: there exists an isoclinic direction at every point of $\mathcal L$.
\end{cor}

\subsection{Holomorphic graphs and $R$-surfaces}\label{Rsurf}
Consider the holomorphic map:
\[
f:U\subset \mathbf C\longrightarrow \mathbf C
\]
Then $\alpha(z)=(z,f(z))$ defines a plane complex curve. We can determine its curvature by:
\[
k(z)=\dfrac{\alpha '(z)\wedge \alpha''(z)}{|\alpha'(z)|^3}=\dfrac{1}{|\alpha'(z)|^3}f''(z)
\]
and its inflection points, where the second derivative $f''(z)$ vanishes, are also the points where the curvature vanishes.

If we consider:
\[
S=\alpha (U)=\hbox{graph }f\subset \mathbf C^2\cong \Rq
\]
we obtain a surface in $\Rq$. Taking $z=x+iy$ and $f(z)=\varphi(x,y)+i\psi(x,y)$, a parametrization is given by
\[
\Xi(x,y)=(x,y,\varphi(x,y),\psi(x,y)
\]
and then the Gaussian curvature is:
\[
K=-\dfrac{2}{(1+\varphi_x^2+\psi_x^2)^3}(\varphi_{xx}^2+\psi_{xx}^2)\le 0
\]
the mean curvature:
\[
\mathcal H\equiv 0
\]
and the normal curvature:
\[
\kappa=\dfrac{2}{(1+\varphi_x^2+\psi_x^2)^3}(\varphi_{xx}^2+\psi_{xx}^2)=-K
\]

The surface $S$ is a Lagrangean surface with respect to two distinct symplectic forms in $\Rq$:
\[
\Omega_1=\de x\wedge \de u - \de y\wedge \de v, \quad \Omega_2=\de x\wedge \de v + \de y\wedge \de u
\]
since:
\[
\Xi^*\Omega_1=\varphi_y \de x\wedge \de y-\psi_x\de y\wedge \de x\equiv 0
\]
and
\[
\Xi^*\Omega_2=\psi_y \de x\wedge \de y+\varphi_x\de y\wedge \de x\equiv 0
\]

The tangent vectors:
\[
T_1=(1,0,\varphi_x,\psi_x),\quad T_2=(0,1,\varphi_y,\psi_y)
\]
are orthogonal and have the same norm, from the Cauchy-Riemann equations. The Gauss map $(x,y)\mapsto e_1\wedge e_2$ is given by:
\[
e_1\wedge e_2=\dfrac{1}{1+\varphi_{x}^2+\psi_{x}^2}(1,-\psi_x, \varphi_x,\varphi_x^2+\psi_x^2,\psi_x,-\varphi_x)
\]
or, in sphere coordinates $(a,b)$, by:
\begin{align*}
a &= \dfrac{1}{1+\varphi_{x}^2+\psi_{x}^2}(1+\varphi_x^2+\psi_x^2,\varphi_x-\varphi_x,\psi_x-\psi_x)=(1,0,0)\\
b &= \dfrac{1}{1+\varphi_{x}^2+\psi_{x}^2}(1-\varphi_x^2-\psi_x^2,\varphi_x+\varphi_x,\psi_x+\psi_x)=\\
&=\dfrac{1}{1+\varphi_{x}^2+\psi_{x}^2}(1-\varphi_x^2-\psi_x^2,2\varphi_x,2\psi_x)
\end{align*}
\begin{defn}
$S$ is a $R$-surface if  it is locally congruent to the real version of a holomorphic graph.
\end{defn}

Thus we see that for $R$-surfaces:
\begin{itemize}
\item All tangent planes are isoclinic to the $(x,y)$ plane, so they all belong to $\mathcal I$, in fact to  $\mathcal I^+$.
\item The image of one component of the Gauss map is a point, $\Gamma_1\equiv (1,0,0)$.
\item  All directions are isoclinic.
\end{itemize}

A straightforward computation shows that the singularities of the Gauss map, where $(x,y)\mapsto b(x,y)$ fails to be an immersion, are given by the inflection points, where  the Gauss curvature vanishes: $\varphi_{xx}=\psi_{xx}=0$.

\section{Inflection points}

The points $p$ where the curvature ellipse passes through the origin are characterised by $\Delta(p)=0$, where:
\beq
\Delta=\dfrac{1}{4}\left|
\begin{matrix}
a&2b&c&0\\
e&2f&g&0\\
0&a&2b&c\\
0&e&2f&g
\end{matrix}
\right|
\eeq
In fact, $\Delta$ is the resultant of the two polynomials $\sff_1=ax^2+2bxy+cy^2$ and $\sff_2=ex^2+2fxy+gy^2$. If there exists $u$ such that $\sff(u)=0$ those polynomials have a common root, and their resultant has to be zero.

A straightforward computation gives:
\begin{align}\label{Delta}
\Delta=&(ac-b^2)(eg-f^2)-\dfrac{1}{4}(ag+ce-2bf)^2=\\
\notag
=&(af-be)(bg-cf)-\dfrac{1}{4}(ag-ce)^2=
\det \mathfrak  N
\end{align}
where:
\[
\mathfrak  N(x,y)=(af-be)x^2+(ag-ce)xy+(bg-cf)y^2
\]
and therefore the equation $\mathfrak  N(x,y)=0$ defining the asymptotic directions \cite{little}, given by $(x,y)$, has two solutions, one or no solutions as $\Delta(p) < 0$, $\Delta(p)=0$ or $\Delta(p) > 0$, respectively.

The points of $S$ may be classified~\cite{little} using $\Delta$, as follows:

\begin{defn}
A pont $p\in S$ is hyperbolic, parabolic or elliptic as  $\Delta(p) < 0$, $\Delta(p)=0$ or $\Delta(p) > 0$, respectively.
\end{defn}

When $\Delta(p)=0$ and the normal curvature vanishes (the curvature ellipse is contained in a line segment),  $\kappa(p)=0$, we can distinguish among the following special points:

\begin{itemize}
\item $\Delta(p) = 0$, $K(p) < 0 $

\noindent
$p$ is an \emph{inflection point
of real type}: the curvature ellipse is a radial segment and $p$ does not belong to it.
 \item $\Delta(p) = 0$, $K(p)=0$

\noindent
$p$ is an \emph{inflection point of flat type}: the curvature ellipse is a radial segment and $p$  belongs to its boundary.
\item $\Delta(p) = 0$, $K(p) > 0$

\noindent
$p$ is an \emph{inflection point of imaginary type}: the curvature ellipse is a radial segment and $p$ belongs to its interior.
\end{itemize}

We have seen before (\ref{EE}) that on a Lagrangean surface we have:
\[
L=H_\psi\quad M=Q,\quad N=H_\varphi
\]
and also (\ref{e=b}):
\[
e\equiv b,\quad f\equiv c\hbox{\ \ and thus\ \ }af-be=ac-b^2,\quad bg-cf=eg-f^2
\]

\begin{prop}
Let $S$ be an isoclinic surface and $p\in S$. Then the following conditions are equivalent:
\begin{enumerate}
\item $p$ is an inflection point.
\item $p$ is a parabolic point, $\Delta(p)=0$, where the normal curvature vanishes, $\kappa(p)=0$.
\item $\rank\left[\begin{matrix}
a&b&c\\
e&f&g
\end{matrix} \right]\le 1$.
\item $p$ is a flat inflection point.
\item $p$ is a parabolic point  where the Gaussian curvature vanishes, $K(p)=0$.
\item $\rank\left[\begin{matrix}
a&b&e&f\\
b&c&f&g
\end{matrix} \right]\le 1$.
\item The Gauss map $\Gamma : S\lra G_{2,4}\subset \mathbf S^5$ is not an immersion at $p$.
\end{enumerate}
\end{prop}
\begin{proof}
It is proved in \cite{little} that the first three statements are equivalent, and the last four are also equivalent and imply the former; but, since $|K|=|\kappa|$, the third and fifth statements are equivalent, therefore they all are equivalent.
\end{proof}

\begin{prop}
If the Gaussian curvature vanishes at a point $p$  in the discriminant curve $\Delta=0$, then the discriminant curve has a non Morse singularity at $p$.
\end{prop}
\begin{proof}
The Gauss curvature of a surface in $\Rq$ is the sum of the Gaussian curvatures of its projections in $\mathbf R^3$ along two orthogonal normal directions, that we can take as being $e_3$ and $e_4$:
\beq
K=K_1+K_2, \quad K_1=ac-b^2, \quad K_2=eg-f^2
\eeq
and if $K(p)=0$ we have:
\[
K_1(p)=-K_2(p)
\]

The equation $\Delta=0$ at the same point gives:
\beq
K_1(p)=0,\quad K_2(p)=0, \quad (ag+ce-2bf)(p)=0\label{cz}
\eeq
and therefore $p$ is a singular point:
\[
\Delta_{x}(p)=0, \quad \Delta_{y}(p)=0
\]

The singularity is non Morse if the Hessian vanishes, $H_{\Delta}(p)=0$; it is proved \cite{dcc,rdcc} that there exists a positive real number $\xi$ such that, at a critical point,  $H_{\Delta}(p)=\xi K(p)$.
\end{proof}

\begin{prop}\label{gfip}
If a point in surface $S$ is a  flat inflection point then, for any projection on $\mathbf R^3$ along a normal direction, its image belongs to the parabolic curve of that projection, where the Gaussian curvature vanishes.
\end{prop}
\begin{proof}
A point  is a flat inflection point if it belongs to the intersection of the discriminant curve $\Delta=0$ and the curve of zero Gaussian curvature $K=0$; we have already seen (\ref{cz}) that $K_1=K_2=0$ at any   point where $\Delta=0$ and $K=0$.

If $K_1=K_2=0$ the projection of the point belongs to the parabolic curve of the surface in $\mathbf R^3$  obtained by projections along the normal directions given by $e_3$ and $e_4$. It is easy to see that the same is true if the direction is given by a linear combination of $e_3$ and $e_4$.
 \end{proof}

Also, if the image belongs to the parabolic curve  for any projection on $\mathbf R^3$ along a normal direction, this is in particular true for the directions given by $e_3$ and $e_4$ and $K_1=K_2=0$.

Our interest in the inflection points in a Lagrangean surface is the study of  the Gauss map around them, therefore we have to make changes of coordinates that are simultaneously isometries and symplectomorphisms, as  they do not affect the metric properties nor the symplectic (or Lagrangean) properties.

\begin{lemma}\label{nfip}
Let $p\in S$ be a point in a Lagrangean  surface, given in normal form as the graph around the origin of a map $(x,y)\mapsto (\varphi(x,y), \psi(x,y))$. If $K_1(p)=0$ and  $K_2(p)=0$, then $p$ is an inflection point, and:
\[j^3\varphi=\dfrac{1}{2}Cx^2+\dfrac{1}{3}ax^3+bx^2y+cxy^2+\dfrac{1}{3}dy^3, \quad j^3\psi=\dfrac{1}{3}bx^3+cx^2y+dxy^2 +e \dfrac{1}{3}y^3
\]
\end{lemma}

\begin{proof}
At $H_{\varphi}(0,0) = H_{\psi}(0,0)=0$ we see that $S$ is the graph of $(x,y)\mapsto (\varphi(x,y), \psi(x,y))$, with:
\[
j^2\varphi=\dfrac{1}{2}A(a_1x+a_2y)^2, \qquad j^2\psi=\dfrac{1}{2}B(b_1x+b_2y)^2
\]
and $\|a\|=\|b\|=1$, where $a=(a_1,a_2)$, $b=(b_1,b_2)$.

From $K(0,0)=H_{\varphi}(0,0) + H_{\psi}(0,0)=0$ it follows that  $0=\kappa (0,0)=L(0,0)+N(0,0)$ and we also have:
\[
\left|
\begin{matrix}
\varphi_{xy}&\psi_{xy}\\
\varphi_{yy}&\psi_{yy}
\end{matrix}
\right|
+
\left|
\begin{matrix}
\varphi_{xx}&\psi_{xx}\\
\varphi_{xy}&\psi_{xy}
\end{matrix}
\right|=0
\]
at the origin. This is readily evaluated as:
\[
AB(a\cdot b)(a\wedge b)=0
\]

In the generic case $AB\neq 0$ and thus either $a=\pm b$ or $a\perp b$; if we assume $a\cdot b=0$, or $b_1=a_2$, $b_2=-a_1$, the conditions:
\[
\varphi_{xy}(0,0)=\psi_{xx}(0,0), \quad  \varphi_{yy}(0,0)=\psi_{xy}(0,0) 
\]
are equivalent to:
\[
Aa_1a_2=Ba_2^2, \quad Aa_2^2=-Ba_1a_2
\]
These are impossible if  we assume the generic condition $\varphi_{xy}(0,0)\neq 0$; under this condition we have $a\wedge b$, or $a=\pm b$ and also:
\[
B=A\dfrac{a_2}{a_1}
\]
so that:
\[
j^2\varphi=\dfrac{1}{2}A(a_1x+a_2y)^2, \qquad j^2\psi=\dfrac{1}{2}A\dfrac{a_2}{a_1}(a_1x+a_2y)^2
\]

The change of coordinates:
\begin{align*}
\xi=a_1x+a_2y \quad  \eta=&a_2x-a_1y\\
\zeta=a_1u+a_2v \quad w=&a_2u-a_1v
\end{align*}
is a symplectomorphism and an isometry, so it belongs to $SU(2)$ and takes isoclinic planes into isoclinic planes. In these coordinates (but writing with the old variables) we get:
\[
j^2\varphi=\dfrac{1}{2}Cx^2, \quad j^2\psi=0, \qquad C=A\left(a_1+\dfrac{a_2^2}{a_1}\right)
\]
\end{proof}

\begin{rem}
If the surface is isoclinic,  it is still true that:
\[
j^2\varphi=\dfrac{1}{2}Cx^2, \quad j^2\psi=0, \qquad C=A\left(a_1+\dfrac{a_2^2}{a_1}\right)
\]
In this case, the situation $a\cdot b=0$ is impossible, because at an inflection point there exists a linear combination of $\varphi$ and $\psi$ with vanishing zero 2-jet \cite{dcc}.
\end{rem}

\begin{theorem}
The existence of a flat inflection point in a generic Lagrangean surface is a stable situation  inside the class of Lagrangean surfaces.
\end{theorem}
\begin{proof}
All inflection points $p$ of a Lagrangean surface $\mathcal L$ are flat inflection points, and we have $\Delta(p)=0 $, $K(p)=0$, and even $K_1(p)=K_2(p)=0$.

Assume now that  $K_1(p)=K_2(p)=0$; then using the  normal form around $p$  of lemma~\ref{nfip}, we get  $\Delta=0$ and $K=K_1+K_2=0$ at $p$. This shows that $p$ is a flat inflection point. 

Locally, a Lagrangean surface $\mathcal L$ is given by the graph around the origin of a map $(x,y)\mapsto (\varphi(x,y), \psi(x,y))$, for which $\varphi_y\equiv \psi_x$. So we can identify $\mathcal L$ with a map $F$ such that:
\[
\varphi=F_x,\quad \psi=F_y
\]

The conditions $K_1=0$, or $K_2=0$, are represented in the 3-jet space $J^3(\Rd,\mathbf R)$ as closed algebraic sets of codimension one, and $K_1=K_2=0$ as a closed algebraic set of codimension two. For a residual set of maps $F$, $j^3F$ intersects transversally those stratified sets, therefore $K_1=0$ and $K_2=0$ are smooth curves intersecting transversally at isolated points corresponding to $K_1=K_2=0$.

Assume that for a given $F$, or equivalently, for a given Lagrangean surface $\mathcal L$ that intersection is non empty: the curves $K_1=0$ and $K_2=0$ are smooth curves intersecting transversally at $p$, or $j^3F$ intersects transversally the algebraic sets representing $K_1=0$, $K_2=0$ and $K_1=K_2=0$; then near $p$ so does the 3-jet of any small perturbation of $F$.
\end{proof}

\subsection{Singularities of the Gauss map}

A generic  map from $\Rd$ into $\Rq$  has no stable singularities. This is true also fo the Gauss map of a generic surface in $\Rq$, since in general the intersection of $\Delta=0$ with $K(p)=0$ is empty \cite{little}.

The situation is quite different for the Gauss map of Lagrangean and isoclinic surfaces: then the
Gauss map  is locally a map from $\Rd$ into $\Rt$ and it can present cross-cap singularities in a stable way \cite{mond}.

\begin{prop}
Let $S$ be an isoclinic surface. Then:
\begin{itemize}
\item the singularities of its Gauss map are inflection points;
\item for a generic isoclinic or Lagrangean surface, the inflection points are cross-cap singularities of the Gauss map.
\end{itemize}
\end{prop}

\begin{proof}
We assume $S$ in Monge form:
\[(x,y)\mapsto (x,y,\varphi(x,y),\psi(x,y))
\]
with the point of interest being the image of the origin, and the tangent plane there being the $(x,y)$-plane, $z=0, w=0$.

We have already seen that the singularities of the Gauss map are flat  inflection points of the surface.
Assume now that $H_{\varphi}(0,0) = H_{\psi}(0,0)=0$.

A parametrization of the Grassmannian $G_{2,4}$ is given by $z=\alpha_1\xi+\beta_1\eta$ and $w=\alpha_2\xi+\beta_2(\eta$, and the Gauss map can be written as:
\[
\alpha(x,y)=(\varphi_{x}(x,y),\psi_{x}(x,y)), \quad \beta(x,y)=(\varphi_{y}(x,y),\psi_{y}(x,y))
\]
The Gauss map is singular when the matrix:
\[
\left[
\begin{matrix}
\varphi_{xx}&\psi_{xx}&\varphi_{xy}&\psi_{xy}\\
\varphi_{yx}&\psi_{yx}&\varphi_{yy}&\psi_{yy}
\end{matrix}
\right]
\]
has rank one (or zero) at the origin. In the generic case we can use the normal form of lemma~\ref{nfip},  or the subsequent remark, and then the matrix above becomes:
\[
\left[
\begin{matrix}
C&0&0&0\\
0&0&0&0
\end{matrix}
\right]
\]
The Gauss map fails to be an immersion, and therefore it has a singularity. For a generic Lagrangean or isoclinic surface $C\ne 0$, so the singularity is a cross-cap.
\end{proof}

\section{The Gauss map for Lagrangean surfaces}

\begin{theorem}
A necessary  and sufficient condition for a surface $S\subset \Rq$ with Gauss map $\Gamma : S\lra G_{2,4}\subset \mathbf S^5$ to be congruent to a Lagrangean surface is that the image of $\Gamma_1$, or  $\Gamma_2$, be contained in a great circle.
\end{theorem}
\begin{proof}
Assume the immersion:
\[
\Xi : U\subset\mathbf R^2\longrightarrow \Rq
\]
defines a Lagrangean surface of $\Rq$ with symplectic form $\omega=\de x\wedge \de u+\de y\wedge \de v$. Then $\Xi^*\omega \equiv 0$ or equivalently:
\[
p_{13}+p_{24} \equiv 0
\]
Recalling the definition of $\Gamma$, we see that its image is contained in $\mathbf S^5\cap \{p_{13}-p_{42}=0\}$, a great circle. In sphere coordinates this is equivalent to $b_2=0$, thus he image of $\Gamma_2$ is contained in a great circle. The other case $a_2=0$) corresponds to the opposite orientation.

Assume instead that the immersion $\Xi$ has a Gauss map such that the image of $\Gamma_2$ is contained in a great circle. This means there exists  a vector $\alpha\in\mathbf S^2\subset\mathbf R^3$ such that $\alpha\cdot b=0$ defines the plane passing through the origin containing the image of $\Gamma_2$. We have to prove that there exists a rigid motion in $\Rq$, taking $S$ into $\hat S$, such that the image of $\hat \Gamma_2$ is contained in $b_2=0$.

It is clear that translations in $\Rq$ do not affect the Gauss map. The linear action of $ {\rm SO}(4)$ on $\Rq$ induces an action on $G_{2,4}$, that can be extended to $\mathbf R^6$ as a linear action:

\begin{lemma}
If $A\in {\rm SO}(4)$ there exists a linear operator $\mathfrak A$ on $\mathbf R^6$ such that if $P=v_1\wedge v_2$ then:
\[
Av_1\wedge Av_2=\mathfrak A P
\]
\end{lemma}
\begin{proof}
Let $A^i$ be the $i$th-column of $A$ in the standard basis, identified with the vector $\sum A^i_k e_k$, and let $\mathfrak A^{ij}=A^i\wedge A^j$ be a column of $\mathfrak A$ (again the vector $\sum A^{ij}_{kl} e_k\wedge e_l$ is the exterior product of the vectors represented by $A^i$ and $A^j$):
\[
A=[A^1:A^2:A^3:A^4],\quad \mathfrak A=[\mathfrak A^{12}:\mathfrak A^{13}:\mathfrak A^{14}:\mathfrak A^{34}:\mathfrak A^{42}:\mathfrak A^{23}]
\]
A simple computation shows that $\mathfrak A$ is the desired linear operator on $\mathbf R^6$; a similar construction can be made based on the lines of the two matrices.
\end{proof}

\begin{lemma}
The action of ${\rm SO}(4)$  on $\mathbf R^6$ is orthogonal:
\[
(A,P)\mapsto  \mathfrak A P, \quad \mathfrak A\in {\rm O}(6)
\]
\end{lemma}
\begin{proof}
The action of  ${\rm SO}(4)$  on $\mathbf R^6$ is defined by:
$
(A,P)\mapsto  \mathfrak A P
$.
This is indeed a well defined action since:
\[
C=AB \Longrightarrow\mathfrak C= \mathfrak A \mathfrak B
\]
where $A,B,C\in {\rm SO}(4)$ and $\mathfrak A$, $\mathfrak B$, $\mathfrak C$ are the corresponding linear operators on $\mathbf R^6$.
Then $\mathfrak A \mathfrak A^T=I$ follows from $AA^T=I$.
\end{proof}

\begin{lemma}\label{l3}
$\mathbf S^5$ and $G_{2,4}$ are invariant for the action induced by ${\rm SO}(4)$:
\[
 \mathfrak A\mathbf S^5\subset\mathbf S^5, \quad \mathfrak A\mathbf G_{2,4}\subset\mathbf G_{2,4}, \quad \hbox{for all }A\in {\rm SO}(4)
\]
Moreover the action is transitive on $G_{2,4}$.
\end{lemma}

\begin{proof}
The invariance of $\mathbf S^5$ follows from the action being orthogonal, and then the invariance of $G_{2,4}$ is a consequence of the action of ${\rm SO}(4)$ on $\Rq$ take planes into planes.

Let $P_1$ and $P_2$  be any two points in $G_{2,4}$, and $\{u_i\}$, $\{v_i\}$, $i=1,\ldots,4$ be orthonormal bases of $\Rq$ with the same orientation such that $u_1,u_2\in P_1$ and $v_1,v_2\in P_2$; there exists $A\in {\rm SO}(4)$ taking one basis into the other, therefore taking the plane $P_1$ into the plane $P_2$ (note they do not need to have the same orientation). Then $\mathfrak A P_1=P_2$.
\end{proof}

Now  $\alpha\cdot b=0$ defines the plane passing through the origin containing the image of $\Gamma_2$; this means that the image of $(\Gamma_1,\Gamma_2)$ is contained in the hyperplane $(0,\alpha)\cdot (a,b)=0$, or that the image of $\Gamma$ is contained in the hyperplane $(\alpha,\alpha)\cdot P=0$.

If there exists $A\in {\rm SO}(4)$ such that:
\[
\mathfrak A (\alpha,\alpha)=\beta, \quad \beta=(0,1,0,0,-1,0)
\]
then, for all  $P$ in the plane  $(\alpha,\alpha)\cdot P=0$:
\[
\beta\cdot (\mathfrak A P)=(\mathfrak A (\alpha,\alpha))\cdot (\mathfrak A P)=(\alpha,\alpha)\cdot P=0
\]
Therefore $\mathfrak A$ takes the plane $(\alpha,\alpha)\cdot P=0$ into the plane $p_{13}-p_{42}=p_{13}+p_{24}=0$, and the surface $A\Xi$ is Lagrangean.

Consider $\hat A\in {\rm SO}(4)$ defined by:
\[
\hat A=
\left[
\begin{array}{cccc}
1&0&0&0\\[4pt]
0&\dfrac{\alpha_2}{\sqrt{\alpha_1^2+\alpha_2^2}}&\alpha_1&\dfrac{\alpha_1\alpha_3}{\sqrt{\alpha_1^2+\alpha_2^2}}\\[16pt]
0&-\dfrac{\alpha_1}{\sqrt{\alpha_1^2+\alpha_2^2}}&\alpha_2&\dfrac{\alpha_2\alpha_3}{\sqrt{\alpha_1^2+\alpha_2^2}}\\[16pt]
0&0&\alpha_3&-\dfrac{(\alpha_1^2+\alpha_2^2)}{\sqrt{\alpha_1^2+\alpha_2^2}}
\end{array}
\right]
\]
Then:
\[
\hat{\mathfrak A}(0,1,0,0,-1,0)=(\alpha,\alpha)
\]
and $A=\hat{A}^{-1}$ gives the required rigid motion.
\end{proof}

\begin{rem}
These lemmas can be extended to arbitrary $n$ and $k$, the action of ${\rm SO}(n)$ on $\mathbf R^n$ inducing an orthogonal action on the Euclidean space $\mathbf R^N$ containing $G_{k,n}$ with the same properties.
\end{rem}

\begin{ex}\label{notsuf0}
Consider the surface $S$ given by:
\[
\varphi(x,y)=x^2-y^2,\quad \psi(x,y)=ax+by-2xy
\]
A straightforward computation shows that:
\[
H_{\varphi}=H_{\psi}=L=N=-4,\quad M=Q=0
\]
therefore:
\[
K=\kappa\Longleftrightarrow E+G=\hat E+\hat G
\]
This is easily verified, as:
\begin{align*}
E=&1+4x^2+(a-2y)^2,& \hat E=&1+4x^2+4y^2\\
G=&1+4y^2+(b-2x)^2,& \hat G=&1+(a-2y)^2+(b-2x)^2
\end{align*}

The surface has Gauss curvature and normal curvature that coincide,  but   if $a$ is not zero we have $\psi_x\not\equiv \varphi_y$, and  if $b$ is not zero we have $\psi_y\not\equiv -\varphi_x$.

On the other hand, the image of $\Gamma_2$ is the curve:
\[
\dfrac{1}{\sqrt{W}}(\sqrt{W-a^2-b^2},b,-a)
\]
which is contained in the great circle $\mathbf S^2\cap \{ab_2+b b_3=0\}$; the image of $\Gamma$ is  contained in the plane:
\[
\alpha\cdot P=0, \quad \alpha=\dfrac{1}{\sqrt{2(a^2+b^2)}}(0,-a,b,0,a,-b)
\]
and is therefore contained in a great circle.

Then there exists a  linear map $A\in {\rm SO}(4)$ in $\Rq$ (lemma~\ref{l3}) that takes the above surface to a Lagrangean surface, in fact:
\[
A=
\left[
\begin{array}{rccl}
1&0&0&0\\
0&1&0&0\\
0&0&-b/\sqrt{a^2+b^2}&-a/\sqrt{a^2+b^2}\\
0&0&a/\sqrt{a^2+b^2}&-b/\sqrt{a^2+b^2}
\end{array}
\right]
\]

The image $\hat S$ of $S$ by $A$ is again parametrized in the form:
\[
(x,y)\mapsto (x,y, \hat\varphi(x,y), \hat\psi(x,y))
\] 
with
\begin{align*}
\hat\varphi(x,y)=&\dfrac{1}{\sqrt{a^2+b^2}}\left(a(2xy-ax-by)-b(x^2-y^2)\right)\\
\hat\psi(x,y)=&\dfrac{1}{\sqrt{a^2+b^2}}\left(b(2xy-ax-by)+a(x^2-y^2)\right)
\end{align*}
Then:
\[
\dfrac{\partial \hat\varphi}{\partial y}(x,y)\equiv \dfrac{\partial \hat\psi}{\partial x}(x,y)
\]
and therefore the surface $\hat S$ is Lagrangean.
\end{ex}

\begin{ex}\label{notsuf}
Consider a surface $S$ such the image of $\Gamma_2$ is the curve:
\[
b_1=c, \quad 0<|c|<1
\]
which is not contained in any great circle.
This condition above is equivalent to:
\[
1+\phi_x\psi_y-\phi_y\psi_x=c\sqrt{1+\phi_x^2+\phi_y^2+\psi_x^2+\psi_y^2+(\phi_x\psi_y-\phi_y\psi_x)^2}
\]

If we fix:
\[
\psi(x,y)=\dfrac{1}{2}(x^2+y^2)
\]
we obtain a first order partial differential equation for $\phi$:
\[
1-y\phi_x+x\phi_y=c\sqrt{1+\phi_x^2+\phi_y^2+x^2+y^2+(y\phi_x-x\phi_y)^2}
\]

We consider the initial condition:
$
\phi(x,0)=-x^2/2
$, and therefore $\phi_x(x,0)=-x$; then $\phi_y(x,0)$ has to satisfy:
\[
1+x\phi_y(x,0)=c\sqrt{1+2x^2+\phi_y(x,0)^2+x^2\phi_y(x,0)^2}
\]

If we choose $c=1/\sqrt{2}$ then at the origin:
\[
2=1+\phi_y(0,0)^2,\quad\hbox{or }\phi_y(0,0)=\pm 1
\]
and it follows from the implicit function theorem that there exists a function $h(x)$ with $h(0)=-1$ such that:
\[
1+xh(x)=c\sqrt{1+2x^2+h(x)^2+x^2h(x)^2}
\]
Then $\phi_y(x,0)=h(x)$ satisfies the required compatibility condition for the initial condition.

We can integrate the partial differential equation in order to obtain $\phi(x,y)$ around the origin, with the above initial condition and choice of $c$, if the origin is a non characteristic point; as the characteristic vector field is:
\[
\dot x =\dfrac{\partial F}{\partial p},\quad 
\dot y =\dfrac{\partial F}{\partial q},\quad \dot  p =-\dfrac{\partial F}{\partial x},\quad \dot q =-\dfrac{\partial F}{\partial y}
\]
where:
\[
p=\phi_x,\quad q=\phi_y
\]
and
\[
F(x,y,p,q)=1-yp+xq-c\sqrt{1+p^2+q^2+x^2+y^2+(yp-xq)^2}
\]
the required transversality condition is satisfied:
\[
\dfrac{\partial F}{\partial q}(0,0,0,-1)=-c\dfrac{-1}{\sqrt{1+(-1)^2}}=\dfrac{1}{2}\ne 0
\]

The Gauss map of the surface $S$ just constructed is non degenerate around the origin, as $\Gamma_1$ is an immersion:
\[
\Gamma_1=\dfrac{1}{\sqrt{1+\phi_x^2+\phi_y^2+x^2+y^2+(y\phi_x-x\phi_y)^2}}(1+y\phi_x-x\phi_y,-\phi_x+y,-\phi_y-x)
\]
and at the origin $\Gamma_1=(\sqrt{2}/2,0,-\sqrt{2}/2)$ and:
\[
\dfrac{\partial}{\partial x}\Gamma_1=(\sqrt{2},\sqrt{2}/2,-\sqrt{2}),\quad \dfrac{\partial}{\partial y}\Gamma_1=(0,-\sqrt{2}/2,0)
\]
taking in account that:
\[
\phi_{xx}(0,0)=-1, \quad \phi_{xy}(0,0)=h'(0)=2,\quad \phi_{yy}(0,0)=0
\]
The value of $\phi_{yy}(0,0)$ can be computed from the fact that the characteristic vector field at $(0,0,0,-1)$ is:
\[
(F_p,F_q,-F_x,-F_y)|_{(0,0,0,-1)}=\left(0,\dfrac{1}{2},1,0\right)
\]

The surface $S$ has Gauss curvature and normal curvature that coincide,  as $0=\Gamma_2^* \sigma =(K-\kappa)\,\omega_1\wedge \omega_2$, and therefore is an isoclinic surface, but  it is not congruent to a Lagrangean surface as neither the image of $\Gamma_1$ nor the image of $\Gamma_2$ is contained in a great circle; this seems to contradict Theorem 1 in \cite{am2}.
\end{ex}


\end{document}